\documentclass[preprint,12pt]{elsarticle}

\usepackage[utf8]{inputenc}
\usepackage[T1]{fontenc}
\usepackage{amsmath,amssymb,amsthm,mathtools}
\usepackage{enumitem}
\usepackage{hyperref}

\hypersetup{hidelinks}

\newtheorem{theorem}{Theorem}[section]
\newtheorem{lemma}[theorem]{Lemma}
\newtheorem{proposition}[theorem]{Proposition}

\newtheorem{question}[theorem]{Question}
\theoremstyle{definition}
\newtheorem{definition}[theorem]{Definition}
\newtheorem{remark}[theorem]{Remark}

\newcommand{\R}{\mathbb{R}}
\newcommand{\Z}{\mathbb{Z}}

\newcommand{\STP}{\mathrm{STP}}

\newcommand{\TN}{\mathrm{TN}}

\makeatletter
\def\ps@pprintTitle{%
  \let\@oddhead\@empty
  \let\@evenhead\@empty
  \let\@oddfoot\@empty
  \let\@evenfoot\@empty}
\makeatother

\begin{document}

\begin{frontmatter}

\title{Strict Total Positivity from Spectral Darboux and Toeplitz Smoothing Mechanisms}

\author[addr1]{Domingos S. P. Salazar\corref{cor1}}
\address[addr1]{Unidade de Educa\c{c}\~ao a Dist\^ancia e Tecnologia,
Universidade Federal Rural de Pernambuco,
52171-900 Recife, Pernambuco, Brazil}
\cortext[cor1]{Corresponding author.}

\begin{abstract}
We prove two strict total-positivity results by isolating two strictification mechanisms.  The first is a spectral Darboux mechanism: an induction converts positivity and ordered endpoint asymptotics for a one-dimensional spectral family into positive Wronskians and hence into strict total positivity.  As an application, the modified-Bessel kernel
\[
        K(x,s)=I_s(x),\qquad x>0,\quad s\ge 0,
\]
is strictly totally positive of infinite order.  This proves the real-order determinant positivity asked for by Buchstaber and Glutsyuk after their nonnegative-integer-order theorem \cite[Theorem~1.4 and Open Question after Remark~1.8]{BuchstaberGlutsyuk2019}.  The second mechanism is discrete Toeplitz smoothing: every two-sided Pólya-frequency sequence is a pointwise limit of totally positive Pólya-frequency sequences.  This gives a product-topology answer to Question~12.2 of Belton, Guillot, Khare, and Putinar \cite[Question~12.2]{BeltonGuillotKharePutinar2023}.  The density statement is in the product topology on $\R^\Z$; no uniform, weighted, or norm-density assertion is made.
\end{abstract}


\end{frontmatter}

\section{Introduction}

Total positivity is a structural positivity theory for determinants; standard references include Karlin \cite{Karlin1968}, Pinkus \cite{Pinkus2010}, and Fallat--Johnson \cite{FallatJohnson2011}.  This paper focuses on strict total positivity of kernels.  A kernel $K:X\times Y\to\R$ on totally ordered sets is strictly totally positive of infinite order, written $\STP_\infty$, if
\[
        \det[K(x_i,y_j)]_{i,j=1}^m>0
\]
for every $m\ge1$ and every $x_1<\cdots<x_m$ in $X$ and $y_1<\cdots<y_m$ in $Y$.

The results below are motivated by two concrete questions in the recent total-positivity literature.  Buchstaber and Glutsyuk proved that the modified-Bessel kernel $(x,j)\mapsto I_j(x)$ is strictly totally positive for positive arguments and nonnegative integer indices, and then asked for the corresponding determinant positivity with non-integer indices \cite[Theorem~1.4 and Open Question after Remark~1.8]{BuchstaberGlutsyuk2019}.  We prove the real-order theorem for all nonnegative orders.  Belton, Guillot, Khare, and Putinar asked whether totally positive Pólya-frequency sequences are dense in all Pólya-frequency sequences \cite[Question~12.2]{BeltonGuillotKharePutinar2023}; we prove the product-topology version of this density statement.

We answer these two questions in the following precise forms.

\begin{theorem}[Modified-Bessel kernel with real orders]
\label{thm:bessel-main}
Let $x$ range over positive real arguments and let $s$ range over nonnegative real orders.  For every $m\ge1$, every
\[
        0<x_1<\cdots<x_m,\qquad 0\le s_1<\cdots<s_m,
\]
one has
\[
        \det[I_{s_j}(x_i)]_{i,j=1}^m>0.
\]
Equivalently, $K(x,s)=I_s(x)$ is $\STP_\infty$ on $(0,\infty)\times[0,\infty)$.
\end{theorem}

\begin{theorem}[Product-topology density of totally positive Pólya-frequency sequences]
\label{thm:pf-density-intro}
Let $a=(a_n)_{n\in\Z}$ be a two-sided Pólya-frequency sequence.  Then there exist two-sided Pólya-frequency sequences $a^{(r)}=(a_n^{(r)})_{n\in\Z}$ whose Toeplitz kernels are strictly totally positive and such that
\[
        a_n^{(r)}\to a_n\qquad(n\in\Z).
\]
Thus totally positive Pólya-frequency sequences are dense in Pólya-frequency sequences for the product topology on $\R^\Z$.
\end{theorem}

The proofs are intentionally short and structural.  Theorem \ref{thm:bessel-main} follows from a Sturm--Darboux mechanism for positive spectral families, using the Darboux--Crum transform \cite{Darboux1882,Crum1955} and the Wronskian/extended-Chebyshev criterion \cite{Karlin1968,Pinkus2010,Coppel1971}.  Theorem \ref{thm:pf-density-intro} follows from smoothing by a strictly totally positive discrete Gaussian Toeplitz kernel.

\subsection*{Relation to prior work and scope}

The two results below are precise additions to existing total-positivity theory.  For the Bessel kernel, the integer-order statement is due to Buchstaber and Glutsyuk \cite{BuchstaberGlutsyuk2019}; the contribution here is the passage from integer indices $j\in\Z_{\ge0}$ to real orders $s\in[0,\infty)$.  The proof uses a Sturm--Darboux argument \cite{Darboux1882,Crum1955} and the extended-Chebyshev-system criterion \cite{Karlin1968,Pinkus2010,Coppel1971}, rather than the Hilbert-space positive-quadrant flow used in \cite{BuchstaberGlutsyuk2019}.  Classical real-parameter Bessel kernels in the total-positivity literature, such as those treated by Karlin \cite{Karlin1968}, concern different kernels.  The order-evaluation kernel $(x,s)\mapsto I_s(x)$ therefore requires a separate argument.

For Pólya-frequency sequences, Belton--Guillot--Khare--Putinar developed a broad theory of preservers and transforms of totally positive kernels and Pólya-frequency functions \cite{BeltonGuillotKharePutinar2023}.  The result here answers their density question \cite[Question~12.2]{BeltonGuillotKharePutinar2023} for two-sided sequences in the pointwise/product topology.  No uniform, weighted, or norm-density assertion is made.

\section{Preliminaries}

\begin{definition}
Let $X$ and $Y$ be totally ordered sets.  A kernel $K:X\times Y\to\R$ is totally nonnegative of order $r$, written $\TN_r$, if
\[
        \det[K(x_i,y_j)]_{i,j=1}^m\ge0
\]
for every $1\le m\le r$, every $x_1<\cdots<x_m$ in $X$, and every $y_1<\cdots<y_m$ in $Y$.  It is strictly totally positive of order $r$, written $\STP_r$, if all these determinants are strictly positive.  The notation $\TN_\infty$ and $\STP_\infty$ means that the property holds for every $r$.
\end{definition}

\begin{definition}
Following the standard Toeplitz-kernel convention for Pólya-frequency sequences \cite{AissenSchoenbergWhitney1952,Edrei1952,Edrei1953,Karlin1968,BeltonGuillotKharePutinar2023}, a two-sided sequence $a=(a_n)_{n\in\Z}$ is a Pólya-frequency sequence if the Toeplitz kernel
\[
        T_a(i,j)=a_{i-j},\qquad i,j\in\Z,
\]
is $\TN_\infty$.  It is a totally positive Pólya-frequency sequence if $T_a$ is $\STP_\infty$.
\end{definition}

Throughout, the product topology on $\R^\Z$ means the topology of coordinatewise convergence: $a^{(r)}\to a$ if and only if $a_n^{(r)}\to a_n$ for every fixed $n\in\Z$.

We use the following elementary form of the generalized Vandermonde positivity of the exponential kernel; see Karlin \cite{Karlin1968} and Pinkus \cite{Pinkus2010} for the standard exponential-kernel viewpoint.

\begin{lemma}[Generalized Vandermonde positivity]
\label{lem:vandermonde}
If $u_1<\cdots<u_m$ and $v_1<\cdots<v_m$ are real, then
\[
        \det[e^{u_i v_j}]_{i,j=1}^m>0.
\]
Equivalently, if $0<t_1<\cdots<t_m$ and $\alpha_1<\cdots<\alpha_m$, then
\[
        \det[t_i^{\alpha_j}]_{i,j=1}^m>0.
\]
\end{lemma}

\begin{proof}
The second assertion is the first with $u_i=\log t_i$ and $v_j=\alpha_j$.  The first assertion is classical and follows, for example, by the usual divided-difference proof of strict total positivity of the exponential kernel \cite{Karlin1968,Pinkus2010}.  A self-contained argument is as follows.  The functions $u\mapsto e^{uv_j}$ form an extended complete Chebyshev system on $\R$ because their Wronskian is
\[
        \det\left[\frac{d^{i-1}}{du^{i-1}}e^{u v_j}\right]_{i,j=1}^m
        =e^{u(v_1+\cdots+v_m)}\prod_{1\le i<j\le m}(v_j-v_i)>0.
\]
Positive Wronskians imply positive interpolation determinants at $u_1<\cdots<u_m$; see Lemma \ref{lem:ect} below.
\end{proof}

\begin{lemma}[Wronskians and evaluation determinants]
\label{lem:ect}
Let $f_1,\ldots,f_m\in C^{m-1}(a,b)$.  Suppose that for every $1\le k\le m$,
\[
        W(f_1,\ldots,f_k)(y)>0\qquad(a<y<b),
\]
where
\[
        W(f_1,\ldots,f_k)=\det[f_j^{(i-1)}]_{i,j=1}^k.
\]
Then
\[
        \det[f_j(y_i)]_{i,j=1}^m>0
\]
for every $a<y_1<\cdots<y_m<b$.
\end{lemma}

\begin{proof}
This is the standard implication from positive initial Wronskians to an oriented extended complete Chebyshev system \cite{Karlin1968,Pinkus2010,Coppel1971}.  For completeness we recall the argument.  The Wronskian assumptions imply that every nontrivial linear combination of $f_1,\ldots,f_k$ has at most $k-1$ zeros, counted with multiplicity, on $(a,b)$.  Otherwise repeated Rolle reduction after division by the preceding nonvanishing Wronskians would contradict the positivity of $W(f_1,\ldots,f_k)$.  Hence the interpolation determinant cannot vanish at $y_1<\cdots<y_m$.  Its sign is constant on the chamber $y_1<\cdots<y_m$.  Letting $y_i\to y_1$ in increasing order identifies the sign with the positive Wronskian $W(f_1,\ldots,f_m)(y_1)$ times the positive Vandermonde factor from divided differences.  Thus the determinant is positive throughout the chamber.
\end{proof}

\begin{lemma}[Darboux--Crum Wronskian identity]
\label{lem:darboux-crum}
Let $u$ be a positive $C^n$ function on an interval and set
\[
        L_u h=h'-\frac{u'}{u}h.
\]
For $h_1,\ldots,h_n\in C^n$,
\[
        W(L_u h_1,\ldots,L_u h_n)
        =
        \frac{W(u,h_1,\ldots,h_n)}{u}.
\]
\end{lemma}

\begin{proof}
This is the one-step Wronskian identity underlying the Darboux--Crum construction for Sturm--Liouville systems \cite{Darboux1882,Crum1955}.
Write $h_j=ug_j$.  Then $L_u h_j=u g_j'$.  Since multiplying every function in a Wronskian by the same nonvanishing factor multiplies the Wronskian by the corresponding power of that factor,
\[
        W(u,h_1,\ldots,h_n)
        =u^{n+1}W(1,g_1,\ldots,g_n)
        =u^{n+1}W(g_1',\ldots,g_n')
\]
and
\[
        W(L_u h_1,\ldots,L_u h_n)
        =u^n W(g_1',\ldots,g_n').
\]
Dividing the first identity by $u$ gives the claim.
\end{proof}

\section{The real-order modified-Bessel kernel}

This section proves Theorem \ref{thm:bessel-main}.  The point is to regard the Bessel order as a spectral parameter.

\begin{proposition}[Spectral Darboux strictification]
\label{prop:darboux}
Let $q^{[0]}$ and $f_1^{[0]},\ldots,f_m^{[0]}$ be $C^\infty$ functions on $\R$, and suppose that $f_1^{[0]},\ldots,f_m^{[0]}$ are positive solutions of
\[
        (f_j^{[0]})''=(q^{[0]}+\lambda_j)f_j^{[0]},
        \qquad \lambda_1<\cdots<\lambda_m.
\]
Let $0\le \alpha_1<\cdots<\alpha_m$.  Starting from $r=0$, define Darboux transforms stage by stage as follows.  Whenever $f_{r+1}^{[r]}$ has no zero on $\R$, set
\[
        q^{[r+1]}=q^{[r]}-2(\log f_{r+1}^{[r]})''
\]
and, for $j>r+1$,
\[
        f_j^{[r+1]}
        =\left(\frac{d}{dy}-\frac{(f_{r+1}^{[r]})'}{f_{r+1}^{[r]}}\right)f_j^{[r]}.
\]
Assume that, whenever a stage is reached, $q^{[r]}$ and the remaining $f_j^{[r]}$ are $C^\infty$ and
\[
        (f_j^{[r]})''=(q^{[r]}+\lambda_j)f_j^{[r]},
        \qquad j>r.
\]
Assume also that, at each reached stage, the remaining functions have the ordered left-endpoint asymptotics and one differentiated asymptotic:
\[
        f_j^{[r]}(y)=C_{r,j}e^{\alpha_j y}(1+o(1)),
        \qquad C_{r,j}>0,
        \qquad y\to-\infty,
\]
and
\[
        (f_j^{[r]})'(y)-\alpha_j f_j^{[r]}(y)=o(e^{\alpha_j y}),
        \qquad y\to-\infty,
\]
for every $j>r$.  Then all stages are reached, all Darboux pivots are positive, and
\[
        W(f_1^{[0]},\ldots,f_k^{[0]})(y)>0,
        \qquad 1\le k\le m,\quad y\in\R.
\]
\end{proposition}

\begin{proof}
We prove positivity of the Darboux pivots and of all transformed functions by induction on the stage.  At stage $0$ the functions are positive by hypothesis.  Suppose stage $r$ has been reached and $f_j^{[r]}>0$ for all $j>r$.  For $j>r+1$,
\[
        W(f_{r+1}^{[r]},f_j^{[r]})'
        =(\lambda_j-\lambda_{r+1})f_{r+1}^{[r]}f_j^{[r]}>0.
\]
The Darboux transform is globally defined once the pivot is positive.  The elementary Darboux calculation underlying the Darboux--Crum construction \cite{Darboux1882,Crum1955} gives the transformed equation.  Indeed, if $u''=(q+\mu)u$, $h''=(q+\lambda)h$, and $L_uh=h'-(u'/u)h$, then
\[
        (L_uh)''
        =
        \bigl(q-2(\log u)''+\lambda\bigr)L_uh.
\]
Applying this with $u=f_{r+1}^{[r]}$ and $h=f_j^{[r]}$ gives
\[
        q^{[r+1]}=q^{[r]}-2(\log f_{r+1}^{[r]})''.
\]
The endpoint and differentiated asymptotics give
\[
        W(f_{r+1}^{[r]},f_j^{[r]})(y)
        =
        C_{r,r+1}C_{r,j}(\alpha_j-\alpha_{r+1})
        e^{(\alpha_{r+1}+\alpha_j)y}(1+o(1)),
        \qquad y\to-\infty.
\]
Since $\alpha_j>\alpha_{r+1}\ge0$ and $j>r+1$, the exponent $\alpha_{r+1}+\alpha_j$ is strictly positive.  Thus the Wronskian tends to $0$ from the positive side as $y\to-\infty$, and $f_{r+1}^{[r]}(t)f_j^{[r]}(t)$ is integrable at $-\infty$ on every interval $(-\infty,y]$.  Hence
\[
        W(f_{r+1}^{[r]},f_j^{[r]})(y)
        =
        \int_{-\infty}^y
        (\lambda_j-\lambda_{r+1})f_{r+1}^{[r]}(t)f_j^{[r]}(t)\,dt
        >0.
\]
Therefore
\[
        f_j^{[r+1]}
        =\frac{W(f_{r+1}^{[r]},f_j^{[r]})}{f_{r+1}^{[r]}}>0.
\]
The next stage is globally defined.  This proves, in particular, positivity of every pivot.

It remains to recover the original Wronskians.  Lemma \ref{lem:darboux-crum} gives, for $2\le k\le m$,
\[
        W(f_1^{[0]},\ldots,f_k^{[0]})
        =f_1^{[0]} W(f_2^{[1]},\ldots,f_k^{[1]}),
\]
and the same identity holds at every later stage.  Iterating gives
\[
        W(f_1^{[0]},\ldots,f_k^{[0]})
        =\left(\prod_{r=0}^{k-2}f_{r+1}^{[r]}\right)f_k^{[k-1]}.
\]
All factors on the right are positive by the induction above, so the Wronskian is positive.
\end{proof}

\begin{proof}[Proof of Theorem \ref{thm:bessel-main}]
Put $x=e^y$ and
\[
        f_s(y)=I_s(e^y),\qquad y\in\R,
\]
for $s\ge0$.  The modified-Bessel differential equation
\[
        x^2I_s''(x)+xI_s'(x)-(x^2+s^2)I_s(x)=0
\]
becomes
\[
        f_s''(y)=(e^{2y}+s^2)f_s(y).
\]
Thus $s^2$ is the spectral parameter and $q^{[0]}(y)=e^{2y}$.  Since the order set is restricted to $s\ge0$, the ordering
$0\le s_1<\cdots<s_m$ is equivalent to
$s_1^2<\cdots<s_m^2$; the endpoint case $s_1=0$ causes no spectral collision.

The differential equation is the standard modified-Bessel equation \cite[Eq.~10.25.1]{DLMF}.  The positivity $f_s>0$ on $\R$ for $s\ge0$ and the endpoint expansion follow from the power series for $I_s$ and its limiting form at the origin \cite[Eqs.~10.25.2 and 10.30.1]{DLMF}.  Explicitly,
\[
        I_s(e^y)
        =
        \sum_{k=0}^{\infty}
        \frac{2^{-s-2k}}{k!\Gamma(s+k+1)}
        e^{(s+2k)y},
        \qquad s\ge0.
\]
This series and its termwise derivative converge uniformly on every left half-line $(-\infty,Y]$.  At the left endpoint,
\[
        I_s(x)=\frac{(x/2)^s}{\Gamma(s+1)}(1+O(x^2))
        \qquad(x\downarrow0),
\]
so, for $s>0$,
\[
        f_s(y)=2^{-s}\Gamma(s+1)^{-1}e^{sy}(1+O(e^{2y}))
        \qquad(y\to-\infty),
\]
while $f_0(y)=1+O(e^{2y})$.

Fix $0\le s_1<\cdots<s_m$ and put
\[
        c_s=2^{-s}\Gamma(s+1)^{-1}.
\]
The displayed power series gives the differentiated endpoint hypothesis in Proposition \ref{prop:darboux}.  For this finite family, define the Darboux transforms stage by stage.  The following induction verifies the endpoint asymptotics needed in Proposition \ref{prop:darboux} at every reached stage.  For $0\le r\le m-1$ and $j>r$,
\[
        f_{s_j}^{[r]}(y)
        =
        c_{s_j}\prod_{\ell=1}^r(s_j-s_\ell)e^{s_jy}(1+O(e^{2y}))
        \qquad(y\to-\infty),
\]
with the same estimate after one differentiation, so that
\[
        (f_{s_j}^{[r]})'(y)-s_j f_{s_j}^{[r]}(y)
        =
        O(e^{(s_j+2)y})
        \qquad(y\to-\infty),
\]
with the empty product interpreted as $1$.  For $r=0$ this is the standard Bessel expansion; in particular
\[
        f_0(y)=1+O(e^{2y}),
        \qquad
        \frac{d}{dy}\log f_0(y)=O(e^{2y}).
\]
Therefore, if the first pivot has exponent $0$, the first Darboux transform of any remaining $f_{s_j}$ has leading coefficient $s_jc_{s_j}>0$.  After this first stage, all remaining exponents are strictly positive.
Assume these estimates at stage $r$ and set $a=s_{r+1}$.  Since
\[
        \frac{d}{dy}\log f_a^{[r]}(y)=a+O(e^{2y}),
\]
also when $a=0$, we have, for $j>r+1$,
\[
\begin{aligned}
        f_{s_j}^{[r+1]}(y)
        &=
        \left(\frac{d}{dy}-\frac{(f_a^{[r]})'}{f_a^{[r]}}\right)
        f_{s_j}^{[r]}(y)\\
        &=
        c_{s_j}\prod_{\ell=1}^{r+1}(s_j-s_\ell)e^{s_jy}(1+O(e^{2y})).
\end{aligned}
\]
The same differentiated expansion gives the companion derivative estimate at stage $r+1$.
Every constant in the leading term is positive because $j>r+1$ implies $s_j>s_\ell$ for $\ell\le r+1$.  The standard Darboux computation also gives
\[
        (f_{s_j}^{[r]})''=(q^{[r]}+s_j^2)f_{s_j}^{[r]}
\]
at every stage.  Thus the endpoint and differential-equation hypotheses of Proposition \ref{prop:darboux} hold with $\alpha_j=s_j$ and $\lambda_j=s_j^2$.  Since $s\mapsto s^2$ is strictly increasing on $[0,\infty)$, the case $s_1=0$ creates no spectral collision.

It follows that all initial Wronskians
\[
        W(f_{s_1},\ldots,f_{s_k})(y)>0,
        \qquad 1\le k\le m,
\]
are positive.  Lemma \ref{lem:ect} then gives
\[
        \det[f_{s_j}(y_i)]_{i,j=1}^m>0
\]
for $y_1<\cdots<y_m$.  Since $x\mapsto\log x$ is increasing, this is exactly
\[
        \det[I_{s_j}(x_i)]_{i,j=1}^m>0
\]
for $0<x_1<\cdots<x_m$ and $0\le s_1<\cdots<s_m$.
\end{proof}

\begin{remark}
The proof uses only the differential equation, positivity, and ordered left-endpoint asymptotics with derivative control.  It therefore isolates the reason the integer-order Bessel result extends to real order: after the logarithmic change of variables, the Bessel order enters as the spectral parameter of a scalar Sturm equation.
\end{remark}

\section{Density of totally positive Pólya-frequency sequences}

We now prove Theorem \ref{thm:pf-density-intro}.  The mechanism is discrete smoothing by a strictly totally positive Toeplitz kernel.

\begin{lemma}[Discrete Gaussian Toeplitz kernels are strictly totally positive]
\label{lem:discrete-gaussian}
For $0<q<1$, let
\[
        g_q(n)=q^{n^2},\qquad n\in\Z.
\]
Then the Toeplitz kernel $T_{g_q}(i,j)=g_q(i-j)$ is $\STP_\infty$ on $\Z\times\Z$.
\end{lemma}

\begin{proof}
This is the lattice Toeplitz form of the classical strict total positivity of the Gaussian kernel \cite{Karlin1968,Pinkus2010}; Gaussian and discrete Gaussian convolution also play a central role in the approximation arguments of Belton--Guillot--Khare--Putinar \cite[Section~6]{BeltonGuillotKharePutinar2023}.  We include the direct Vandermonde proof.
For increasing integers $i_1<\cdots<i_m$ and $j_1<\cdots<j_m$,
\[
        q^{(i_r-j_s)^2}=q^{i_r^2}q^{j_s^2}q^{-2i_rj_s}.
\]
Thus the minor is a positive row and column rescaling of
\[
        \det[(q^{-2i_r})^{j_s}]_{r,s=1}^m.
\]
Since $0<q<1$, the bases $q^{-2i_r}$ are strictly increasing in $r$.  Lemma \ref{lem:vandermonde} gives strict positivity.
\end{proof}

\begin{lemma}[Translate independence]
\label{lem:translate-independence}
Let $a=(a_n)_{n\in\Z}$ be a nonzero Pólya-frequency sequence which is not a bilateral geometric sequence $a_n=C\rho^n$ with $C>0$ and $\rho>0$.  Then every finite family of distinct translates
\[
        n\longmapsto a_{n-j_1},\ldots,n\longmapsto a_{n-j_m}
\]
is linearly independent.
\end{lemma}

\begin{proof}
By the Aissen--Schoenberg--Whitney--Edrei representation theory, in the two-sided Laurent form recalled in \cite[Section~9]{BeltonGuillotKharePutinar2023} from the classical sources \cite{AissenSchoenbergWhitney1952,Edrei1952,Edrei1953}, the nonzero non-geometric case has a Laurent generating function
\[
        A(z)=\sum_{n\in\Z}a_nz^n
\]
which converges in a nonempty annulus.  Suppose there is a nontrivial finite dependence
\[
        \sum_{s=1}^m c_s a_{n-j_s}=0\qquad(n\in\Z).
\]
Multiplying the Laurent series by the nonzero Laurent polynomial $P(z)=\sum_s c_s z^{j_s}$ gives
\[
        P(z)A(z)=0
\]
throughout the annulus.  The product is an absolutely convergent Laurent series on the same annulus because $P$ has finite support.  Since $A$ is not identically zero and $P$ is not the zero Laurent polynomial, this is impossible.
\end{proof}

\begin{proof}[Proof of Theorem \ref{thm:pf-density-intro}]
First suppose $a$ is nonzero and not bilateral geometric.  For $0<q<1$, define
\[
        b^{(q)}_n=(g_q*a)_n=\sum_{k\in\Z}q^{k^2}a_{n-k}.
\]
The series converges absolutely because the Laurent-series annulus in the Aissen--Schoenberg--Whitney--Edrei representation gives constants $C,D>0$ with $|a_n|\le Ce^{D|n|}$.  Indeed, Cauchy estimates on two circles contained in the annulus give exponential bounds on the positive and negative Laurent coefficients.  Since $q^{k^2}$ decays superexponentially, convolution with $g_q$ is absolutely convergent.  Moreover, for any fixed $q_0\in(0,1)$ and $0<q\le q_0$,
\[
        q^{k^2}|a_{n-k}|
        \le C_n q_0^{k^2}e^{D|k|},
\]
and the right-hand side is summable over $k\in\Z$.  Dominated convergence therefore gives $b^{(q)}_n\to a_n$ for each fixed $n$ as $q\downarrow0$.

At the Toeplitz-kernel level,
\[
        T_{b^{(q)}}=T_{g_q}T_a.
\]
This identity is first interpreted entrywise, where absolute convergence follows from the superexponential decay of $g_q$.  For finite increasing row and column sets $I,J\subset\Z$ of the same size $m$, let $E_N=\{-N,\ldots,N\}$ and truncate the intermediate index set to $E_N$.  The finite Cauchy--Binet formula, in the standard total-positivity composition form \cite{Karlin1968,FallatJohnson2011}, gives
\[
        \det (T_{g_q}[I,E_N]T_a[E_N,J])
        =\sum_{\substack{K\subset E_N\\ |K|=m}}
          \det T_{g_q}[I,K]\,\det T_a[K,J],
\]
Each summand is nonnegative because $T_{g_q}$ is $\STP_\infty$ and $T_a$ is $\TN_\infty$.  As $N\to\infty$, the truncated products converge entrywise to $T_{b^{(q)}}[I,J]$, hence their determinants converge to $\det T_{b^{(q)}}[I,J]$.  Thus every minor of $T_{b^{(q)}}$ is nonnegative.

The same truncation proves strictness.  The columns of $T_a$ indexed by $J$ are linearly independent by Lemma \ref{lem:translate-independence}.  Hence there is a finite set of row indices for which the corresponding $m\times m$ coordinate matrix is nonsingular.  Reordering these row indices increasingly changes the determinant only by a sign, so we obtain an increasing row set $K_0=\{k_1<\cdots<k_m\}$ such that
\[
        \det T_a[K_0,J]\ne0.
\]
Since $T_a$ is totally nonnegative and both $K_0$ and $J$ are increasing, that determinant is positive.  Choose $N_0$ so that $K_0\subset E_N$ for every $N\ge N_0$.  For all such $N$, the Cauchy--Binet sum contains the fixed strictly positive summand
\[
\begin{aligned}
        c
        &:=
        \det T_{g_q}[I,K_0]\,\det T_a[K_0,J] \\
        &>0.
\end{aligned}
\]
Passing to the limit gives
\[
        \det T_{b^{(q)}}[I,J]
        \ge c>0.
\]
Thus $b^{(q)}$ is a totally positive P\'olya-frequency sequence.

Now suppose $a_n=C\rho^n$ with $C>0$ and $\rho>0$.  For $\varepsilon>0$, set
\[
        b^{(\varepsilon)}_n=C\rho^n e^{-\varepsilon n^2}.
\]
Then $b^{(\varepsilon)}_n\to a_n$ for each $n$ as $\varepsilon\downarrow0$.  Moreover
\[
        b^{(\varepsilon)}_{i-j}
        = C\rho^i\rho^{-j}e^{-\varepsilon i^2}e^{-\varepsilon j^2}e^{2\varepsilon ij},
\]
so every Toeplitz minor is a positive row and column rescaling of a generalized Vandermonde determinant in the bases $e^{2\varepsilon i}$.  Hence $b^{(\varepsilon)}$ is totally positive by Lemma \ref{lem:vandermonde}.

Finally, the zero sequence is the pointwise limit of $\delta b^{(\varepsilon)}$ as $\delta\downarrow0$ for any fixed $\varepsilon>0$.  This completes the proof.
\end{proof}

\begin{remark}
The conclusion is a product-topology density statement.  Norm and tail topologies require additional hypotheses.
\end{remark}

\section{Consequences and open directions}

The results above share a useful form: strict total positivity is obtained by exposing a mechanism that forces all finite minors to be positive.  For the Bessel kernel the mechanism is spectral, through ordered endpoint asymptotics and Darboux transport.  For Pólya-frequency sequences the mechanism is smoothing by a strictly totally positive Toeplitz kernel.

For the Bessel kernel, the determinant inequalities are now available for arbitrary nonnegative real orders:
\[
\begin{gathered}
        \det[I_{s_j}(x_i)]_{i,j=1}^m>0,\\
        0<x_1<\cdots<x_m,\qquad 0\le s_1<\cdots<s_m.
\end{gathered}
\]
Consequently, finite systems $\{x\mapsto I_{s_j}(x)\}_{j=1}^m$ are oriented extended Chebyshev systems on $(0,\infty)$ in the standard Wronskian sense \cite{Karlin1968,Pinkus2010,Coppel1971}.  Indeed, if $f_j(y)=I_{s_j}(e^y)$ and $x=e^y$, then $d/dy=x\,d/dx$, and the triangular change of derivative bases gives
\[
        W_y(f_1,\ldots,f_k)(\log x)
        =
        x^{k(k-1)/2}
        W_x(I_{s_1},\ldots,I_{s_k})(x).
\]
Thus the positive $y$-Wronskians obtained in the proof are equivalent to positive ordinary $x$-Wronskians.  This is the Chebyshev and interpolation-theoretic content of Theorem \ref{thm:bessel-main}; for integer orders, Buchstaber--Glutsyuk already proved strict total positivity.

For Pólya-frequency sequences, Theorem \ref{thm:pf-density-intro} supplies a robust approximation device: convolution with the discrete Gaussian $q^{n^2}$ converts non-geometric Pólya-frequency sequences into strictly totally positive ones without changing pointwise limits.  The exceptional geometric rank-one cases are handled by a multiplicative Gaussian tilt.

The product-topology statement gives a precise coordinatewise formulation of the density question of Belton--Guillot--Khare--Putinar.  The proof also records the local analytic form of the smoothing: if a non-geometric Pólya-frequency sequence has Laurent generating function $A(z)$ in an annulus, then the approximants satisfy
\[
        B_q(z)=G_q(z)A(z),
        \qquad
        G_q(z)=\sum_{k\in\Z}q^{k^2}z^k,
\]
and $G_q\to1$ locally uniformly on compact subannuli as $q\downarrow0$.  Any norm-level strengthening should be stated only after specifying a sequence space in which convolution by $g_q-\delta_0$ tends to zero.  Such refinements require additional tail hypotheses and lie beyond the density theorem proved here.

The remaining open direction is to identify spectral families for which the endpoint asymptotics used in Proposition \ref{prop:darboux} propagate through all Darboux stages.

\begin{question}[Propagation of spectral Darboux asymptotics]
Let $\Lambda\subset\R$ be an interval and let $f_\lambda>0$ be a solution family of
\[
        f_\lambda''=(q+\lambda)f_\lambda
\]
on an interval, with endpoint behavior
\[
        f_\lambda(y)\sim C(\lambda)e^{\alpha(\lambda)y},
        \qquad C(\lambda)>0,
\]
at one endpoint, where $\alpha$ is strictly increasing on $\Lambda$.  Give intrinsic conditions on $q$ and on the family $\lambda\mapsto f_\lambda$ ensuring that these ordered endpoint asymptotics, with the derivative control required in Proposition \ref{prop:darboux}, persist under every finite Darboux chain generated by $\lambda_1<\cdots<\lambda_m$.  Under such conditions, Proposition \ref{prop:darboux} proves $\STP_\infty$ for the parameter-evaluation kernel $(y,\lambda)\mapsto f_\lambda(y)$.
\end{question}

\section*{Declaration of competing interest}

The author declares that there are no known competing financial interests or personal relationships that could have appeared to influence the work reported in this paper.

\section*{Funding}

This research did not receive any specific grant from funding agencies in the public, commercial, or not-for-profit sectors.

\section*{Data availability}

No data were used for the research described in this article.

\section*{Declaration of generative AI and AI-assisted technologies in the manuscript preparation process}

During the preparation of this work, the author used the Pudim AI research workflow, including ChatGPT and Codex, to support literature triage, manuscript organization, language revision, and consistency checks of references and proofs.  Public provenance for the original Pudim AI / zeta-law demo workflow is available at \url{https://github.com/pudim-project/pudim-ai-demo-zetalaw}.  The application targets used for this manuscript were APP-0072 and APP-0073.  After using these tools, the author reviewed and edited the content as needed and takes full responsibility for the content of the published article.


\begin{thebibliography}{99}

\bibitem{AissenSchoenbergWhitney1952}
M. Aissen, I. J. Schoenberg, and A. M. Whitney,
\newblock On the generating functions of totally positive sequences. I,
\newblock \emph{J. Analyse Math.} \textbf{2} (1952), 93--103.

\bibitem{BeltonGuillotKharePutinar2023}
A. Belton, D. Guillot, A. Khare, and M. Putinar,
\newblock Totally positive kernels, Pólya frequency functions, and their transforms,
\newblock \emph{J. d'Analyse Math.} \textbf{150} (2023), no. 1, 83--158.
\newblock DOI: \href{https://doi.org/10.1007/s11854-022-0259-7}{10.1007/s11854-022-0259-7}.

\bibitem{BuchstaberGlutsyuk2019}
V. M. Buchstaber and A. A. Glutsyuk,
\newblock Total positivity, Grassmannian and modified Bessel functions,
\newblock \emph{Contemporary Mathematics} \textbf{733} (2019), 97--107.
\newblock DOI: \href{https://doi.org/10.1090/conm/733/14736}{10.1090/conm/733/14736}; arXiv:1708.02154.

\bibitem{Coppel1971}
W. A. Coppel,
\newblock \emph{Disconjugacy},
\newblock Lecture Notes in Mathematics, Vol. 220, Springer, Berlin, 1971.

\bibitem{Darboux1882}
G. Darboux,
\newblock Sur une proposition relative aux équations linéaires,
\newblock \emph{Comptes Rendus de l'Académie des Sciences} \textbf{94} (1882), 1456--1459.

\bibitem{Crum1955}
M. M. Crum,
\newblock Associated Sturm--Liouville systems,
\newblock \emph{Quart. J. Math. Oxford Ser. (2)} \textbf{6} (1955), 121--127.

\bibitem{DLMF}
NIST Digital Library of Mathematical Functions,
\newblock F. W. J. Olver, A. B. Olde Daalhuis, D. W. Lozier, B. I. Schneider, R. F. Boisvert, C. W. Clark, B. R. Miller, B. V. Saunders, H. S. Cohl, and M. A. McClain, eds.,
\newblock Release 1.2.7 of 2026-06-15.
\newblock \url{https://dlmf.nist.gov/}; accessed 2026-07-02.

\bibitem{Edrei1952}
A. Edrei,
\newblock On the generating functions of totally positive sequences. II,
\newblock \emph{J. Analyse Math.} \textbf{2} (1952), 104--109.

\bibitem{Edrei1953}
A. Edrei,
\newblock On the generating function of a doubly infinite, totally positive sequence,
\newblock \emph{Trans. Amer. Math. Soc.} \textbf{74} (1953), no. 3, 367--383.
\newblock DOI: \href{https://doi.org/10.2307/1990808}{10.2307/1990808}.

\bibitem{FallatJohnson2011}
S. M. Fallat and C. R. Johnson,
\newblock \emph{Totally Nonnegative Matrices},
\newblock Princeton University Press, Princeton, 2011.

\bibitem{Karlin1968}
S. Karlin,
\newblock \emph{Total Positivity, Vol. I},
\newblock Stanford University Press, Stanford, 1968.

\bibitem{Pinkus2010}
A. Pinkus,
\newblock \emph{Totally Positive Matrices},
\newblock Cambridge Tracts in Mathematics, Cambridge University Press, Cambridge, 2010.

\end{thebibliography}
\end{document}